\documentclass[11pt]{amsart}

\usepackage{bm}
\usepackage{fullpage}
\usepackage{amssymb}
\usepackage{amsmath,amsfonts,amsthm}
\usepackage{hyperref}
\usepackage{color}
\usepackage{cite}

\newtheorem{theorem}{Theorem}

\newtheorem{lemma}[theorem]{Lemma}

\newtheorem{definition}[theorem]{Definition}

\newtheorem{proposition}[theorem]{Proposition}
\newtheorem{corollary}[theorem]{Corollary}

\DeclareMathOperator\G{\mathcal{G}}

\DeclareMathOperator\C{\mathcal{C}}

\title{Nordhaus-Gaddum problem in term of $G$-free coloring
}

\author{Yaser Rowshan$^1$}
\address{$^1$Y. Rowshan, 
	Department of Mathematics, Institute for Advanced Studies in Basic Sciences (IASBS), Zanjan 45137-66731, Iran}
\email{y.rowshan@iasbs.ac.ir}

\begin{document}
	\maketitle 
	
	\begin{abstract} 
		Let $H=(V(H),E(H))$ be a graph. A  $k$-coloring of $H$ is a mapping $\pi : V(H) \longrightarrow \{1,2,\ldots, k\}$, if  each color class induces a $K_2$-free subgraph. For a  graph $G$ of order at least $2$, a $G$-free $k$-coloring of $H$,
		is  a mapping $\pi : V(H) \longrightarrow \{1,2,\ldots,k\}$, so that 
		the   induced subgraph by each color class of $\pi$, contains no  copy of $G$. The $G$-free chromatic number of $H$, is the minimum number $k$, so that  it has a $G$-free $k$-coloring, and denoted by
		$\chi_G(H)$.  In this paper, we give some  bounds and attributes  on the $G$-free chromatic number of graphs, in terms of  the number of vertices, maximum degree,  minimum degree, and chromatic number. Our main results are  the Nordhaus-Gaddum-type theorem for the $\G$-free chromatic number of a graph.
	 
	\end{abstract}
	
	\section{Introduction} 
	Throughout this paper, all graphs are simple, undirected, and finite. For  given graph $G$, the vertex set, edge set, maximum degree, and minimum degree of $G$, denoted by $V(G)$, $E(G)$, $\Delta(G)$, and $\delta(G)$, respectively. The number  of vertices of 
	$G$ is denoted by $|V(G)|$.
	For a vertex $v\in V(G)$, let $\deg_G{(v)}$ ($\deg{(v)}$) and $N_G(v)$ denote the degree and neighbors of $v$ in $G$, respectively. 
	Let $W$ be any subset of  $V(G)$, the induced
	subgraph $G[W]$, is the graph whose vertex set is $W$ and whose edge set consists of all of the edges in $E(G)$ that have both endpoints in $W$.
	The join of two graphs $G$ and $H$, denoted by $G+H$, is a graph obtained from $G$ and $H$ by joining each vertex of $G$ to all vertices of $H$.

	\subsection{$G$-free coloring.}	\hfilneg
	
	The conditional chromatic number $\chi(H,P)$ of $H$,  is the smallest  integer  $k$, such that there 
	exists a decomposition of  $V(H)$  into  sets $V_1,\ldots, V_k$, so that  $H[V_i]$ satisfies the property $P$, where   $P$ is a graphical property, and $H[V_i]$ is the induced subgraph on $V_i$, for each $1\leq i\leq k$.
	This extension of graph coloring was presented by Harary in 1985~\cite{MR778402}. Suppose that $\G$
	be a family of  graphs, when $P$ is the feature that a subgraph induced by each color class does not contain  a  copy of 
	each member of $\G$, we write $\chi_{\G}(H)$ instead of $\chi(H, P)$. In this regard, we say a graph $H$ has a $\G$-free $k$-coloring, if there exists a  map $\pi : V(H) \longrightarrow \{1,2,\ldots,k\}$, such that  each color class  $V_i=\pi^{-1}{(i)}$  does not contain any members of $\G$. For simplicity of notation if $\G=\{G\}$, then we write $\chi_G(H)$ instead of $\chi_{\G}(H)$.
	
	An ordinary  $k$-coloring of $H$  can be viewed as $\G$-free $k$-coloring of a graph $H$ by taking $\G=\{K_2\}$.  It has been shown that for each  graph $H$, $\chi(H)\leq  \Delta(H)+1$. The well-known Brooks theorem, states that   for any  connected graph $H$,  if $H$  is neither an odd $C_n$ nor a $K_n$, then $\chi(H)\leq  \Delta(H)$ \cite{Brooks}.

	%============================================
	The Nordhaus-Gaddum problem \cite{nordhaus1956complementary}, associates with the parameter $f(G)$ of a graph $H$, and asks to find sharp bound for $f(H) + f(\overline{H})$ and $f(H)f(\overline{H})$. Many authors have studied the Nordhaus-Gaddum problem, associate with the parameter $f(H)$.
	%%%%%%%%%%%%%%%%%%%%%%%%%%%
	Nordhaus and Gaddum in \cite{nordhaus1956complementary} investigate the   $\chi(H)$ and  $\chi(\overline{H})$ together, and they proved that if $H$ be a graph with $n$ members, then:
\begin{itemize}
\item 	$ 2\sqrt{|V(H)|}\leq \chi(H)+\chi(\overline{H})\leq |V(H)|+1.$
\end{itemize}
After proving this theorem, inequalities containing the sum or product of a graph and its complement are  called Nordhaus-Gaddum-type inequalities.
	%%%%%%%%%%%%%%%%%%%%%%%%%%%%%%%%%%%%%%%%%%%%%%%%%%%%%%%%%%%%	
The Nordhaus-Gaddum problem associated with the parameter $\chi_n(H)$ (The n-defective chromatic number), is to
find sharp bounds for  $\chi_n(H)+\chi_n(\overline{H})$. Maddox \cite{maddox1989vertex} investigated this problem and showed that if $H$ be a $K_3$-free graph,  then:
\begin{itemize}
	\item $\chi_n(H)+\chi_n(\overline{H})\leq 5\lceil \frac{|V(H)|}{3n+3}\rceil.$
\end{itemize}
%%%%%%%%%%%%%%%%%%%%%%%%%%%%%%%%%%%%%%%%%%%%%%%%%%%%%%%%%%%%	
The domination number of a graph $H$, denoted by $\gamma(H)$, is the cardinality of the minimum dominating set. The general Nordhaus-Gaddum upper bounds  for $\gamma(H)$ is given in \cite{jaeger1972relations} and it has proven that if  $H$ be a graph  whit $n$ members, then:
\begin{itemize}
	\item $3\leq \gamma(H)+\gamma(\overline{H})\leq |V(H)|+1.$
\end{itemize}

%%%%%%%%%%%%%%%%%%%%%%%%%%%%%
One can refer to\cite{aouchiche2013survey} and \cite{brown2009nordhaus, mitchem1971point}, and their references for further studies about the Nordhaus-Gaddum-type theorem.	
	%==========================================
	The vertex  arboricity of graph $H$, denoted by $a(H)$, and defined as the minimum number of colors which are  needed to color the vertices of $H$, so that no cycle is monochromatic.	The general Nordhaus-Gaddum upper bounds for  vertex arboricity is  given as follow:
	\begin{theorem}\label{tho1}\cite{mitchem1971point}Let $H$ is a graph, then:
		\[ \lceil \sqrt{|V(H)|} \rceil \leq a(H)+a(\overline{H})\leq \frac{|V(H)|+3}{2}.\]
	\end{theorem} 
	%%%%%%%%%%%%%%%%%%%%%%%%%%%%%%%%%%%%%%%%%%%%%%
	%==========================================
 In this article, we prove some results for $\chi_{\G}(H)$. Our main results are  the Nordhaus-Gaddum-type theorem for the $\G$-free chromatic number of a graph as follows.
	\begin{theorem}\label{one-thm1}
		Suppose that $\G$
		is a family of  graphs with  minimum degrees  $\delta(\G)$, where $\delta(\G)=\min\{\delta(G):~G\in\G \}$. Also, assume that $H$ is a  graph with $n(H)$ vertices.  Then, we have:
		
		\[
		\chi_{\G}(H)+ \chi_{\G}(\overline{H})\leq\left\lbrace
		\begin{array}{ll}
			
			\lceil\frac{n(H)}{\delta(\G)}\rceil+1 & ~~~~
		when ~\G=\C ~or	~either,~~ H~ or~\overline{H}~is ~\G-free~ critical, ~\vspace{.2 cm}\\
			\lceil\frac{n(H)}{\delta(G)}\rceil+2 & ~~~~ ~~~~~~ otherwise.
		\end{array}
		\right.
		\]
		
	\end{theorem} 
	%%%%%%%%%%%%%%%%%%%%%%%%%%%%%%%%%%%%%%
	%%%%%%%%%%%%%%%%%%%%%%%%%%%%%%%%%%%%%%
	
	In Theorem~\ref{one-thm1}, if we take 
	$\G=\{K_2\}$, then we get  Nordhaus-Gaddum's result. Also,  if we take 
	$\G=\C$,  then we get Theorem~\ref{tho1}.
	%%%%%%%%%%%%%%%%%%%%%%%%%%%%%%%%%%%%%%%%%%%%%%%%%%%%%%%%%%%

	\section{Main results}

To prove Theorem \ref{one-thm1},  we need  some theorems and lemmas. In this section,  we give some properties of  $G$-free coloring, and  some upper  bounds on the $G$-free chromatic number of graphs.  
\subsection{Some properties  and some bounds of  $G$-free coloring of graphs.}\hfilneg
	
	$k$-critical graphs have a key role in studying ordinary graph coloring. Here we define the $G$-free $k$-critical graphs.
	%---------------------------------------------------------------------------------------%
	\begin{definition}
		A graph $H$ with $\chi_G(H)=k$ is said  $G$-free $k$-critical,   if for each proper subgraph  $H'$ of   $H$, $\chi_G(H')\leq k-1$. 
	\end{definition}
	%---------------------------------------------------------------------------------------%
	Let $H$ be a graph with $\chi_G(H)=k$. By taking a minimal subgraph of $H$ with $G$-free chromatic number $k$,
	it is easy to say that this subgraph is   $G$-free $k$-critical.
	It is well known  that if $H$ be a $k$-critical graph with $\chi(H)=k$, then $\delta(H)\geq \chi(H)-1$. As a generalization, we have the following result:
	
	\begin{lemma}\label{the2} If $H$ is a $G$-free $k$-critical graph, then: 
		\[\delta(H)\geq \delta(G) (k-1).\]
	\end{lemma}
	%---------------------------------------------------------------------------------------%
	\begin{proof}
		By contradiction, let there exists a vertex of $V(H)$ say $x$, such that $\deg(x)\leq\delta(G)(k-1)-1$. Since $H$ is $G$-free $k$-critical, so $\chi_G(H\setminus x)\leq k-1$. Set $H'= H\setminus x$. Without loss of generality(w.l.g)
		suppose that $V_1, V_2,\ldots, V_{k-1}$ is a partition  of $V(H')$, so that $H'[V_i]$ is $G$-free. Since $\deg(x)\leq\delta(G)(k-1)-1$, so there is at least one $j\in[k-1]$, such that  $|N(x)\cap V_j|\leq \delta(G)-1$.
		Hence $H[V_j\cup\{x\}]$ is a $G$-free subgraph of $H$. Thus, $\chi_G(H)\leq k-1$, a contradiction.
	\end{proof}
	%--------------------------------------------------------------------------- 	
	 Now we present some upper bounds on  $\chi_G(H)$.
	In~\cite{szekeres1968inequality} Szekeres and Wilf have shown that $\chi(H)\leq 1+\max \delta(H')$, where $H'\subseteq H$. In the next theorem, we  extend  the Szekeres and Wilf theorem. 
	\begin{theorem}\label{the3} Let $G$ be a given graph. For an arbitrary graph $H$, we have:
		\[\chi_G(H)\leq 1+\lceil \frac{\max \delta(H')}{\delta(G)}\rceil.\]
		where the maximum is taken over all induced subgraphs $H'\subseteq H$. 
	\end{theorem}
	%---------------------------------------------------------------------------------------%
	%---------------------------------------------------------------------------------------%

	%---------------------------------------------------------------------------------------%

	To prove the  next corollary,   we  need  the following  theorem by Lov{\'a}sz. 
	\begin{theorem}\label{th1}\cite{Lovasz} If $d_i\geq d_{i+1}$ for $1\leq i\leq k$, where $k$ is a positive integer,  such that $\sum \limits^{i=k}_{i=1} d_i \geq \Delta(G)-k+1$, then $V(G)$ can be decomposed into  $k$ classes $V_i$($~i=1,\ldots,k$), such that $\Delta(G[V_i])\leq d_i$ for each $i=1,2,\ldots,k$.
	\end{theorem}
	\begin{corollary}\label{col2} Let $H$ and $G$ be two graphs with maximum degrees $\Delta(H)$ and $\Delta(G)$, respectively. Then:
		\[\chi_G(H)\leq\lceil \frac{\Delta(H)+1}{\Delta(G)}\rceil.\]
	\end{corollary}
	%---------------------------------------------------------------------------------------%
	\begin{proof}
		It is easy to show that $\chi_G(H)\leq \chi_{K_2}(H)= \chi(H)\leq \Delta(H)+1$. Hence, if $\Delta(G)=1$, then the corollary holds. 
		Let $\Delta(G)\geq 2$. If $\Delta(H)\leq \Delta(G)-1$, then $H$  has no copy of $G$ and $\chi_G(H)=1$. Assume that
		$\Delta(H)=k\Delta(G)+r$, where $k\geq 1$ and $r\in[\Delta-1]$. For $1\leq i\leq k+1$, we set $d_i=\Delta(G)-1$. Hence:  
		\begin{align*}
			d_1+d_2+\ldots+d_{k+1}&=(k+1)(\Delta(G)-1)\\
			&\geq k\Delta(G)-k+r\\
			&=\Delta(H)-(k+1)+1.
		\end{align*}
		By Theorem~\ref{th1}, $V(H)$ can be decomposed into classes $V_1,V_2,\ldots,V_{k+1}$, such that $\Delta(H[V_i])\leq\Delta(G)-1$ for each $i\in[k+1]$, and as a consequence $H[V_i]$ is  $G$-free. Therefore:
		\[\chi_G(H)\leq k+1=\lceil \frac{\Delta(H)+1}{\Delta(G)}\rceil.\]
	\end{proof}
	%---------------------------------------------------------------------------------------%
	It should be mentioned that the same bound was proved for the defective chromatic number of graphs, by Frick and Henning.
	%---------------------------------------------------------------------------------------%
	\begin{proposition}\label{prop1} Let $H$ and $G$  be two graphs. Then:
		\[\chi_G(H)\leq \lceil \frac{\chi(H)}{\chi(G)-1} \rceil.\]
	\end{proposition}
	%---------------------------------------------------------------------------------------%
	Note that this bound is sharp  for $G\cong H$, $G\cong K_{n+1}$ and $H\cong K_{kn+1}$ or $G=K_2$ and $H$ is an odd cycle or $K_n$.
	%---------------------------------------------------------------------------------------%

	%---------------------------------------------------------------------------------------%

	\begin{lemma}\label{l1}
		For each graph $H$ with $\chi_G(H)=k$, $H$ has a subgraph say $F$, such that  $F$ is a $G$-free $k$-critical graph
	\end{lemma}
	%---------------------------------------------------------------------------------------%
	\begin{lemma}\label{le2}
		For any integers $m,n$, where $n\leq m$, and any two graphs $H$ and $G$ with $\chi_G(H)=m$, $H$ has a subgraph say $F$, such that  $\chi_G(F)=n$.
	\end{lemma}
	\begin{proof}  For $n=0$, the result is trivial. Suppose that $1\leq n\leq m$, and $\chi_G(H)=m$. By Lemma \ref{l1}, $H$ has a subgraph say $F$, such that  $\chi_G(F)=m$. If $n=m$, the result holds, otherwise let $v\in V(F)$, hence  $\chi_G(F\setminus \{v\})=m-1$. So by induction, there is a $G$-free $n$-critical subgraph of both $F\setminus \{v\}$ and  $H$, hence the proof is complete. 
	\end{proof}
	%---------------------------------------------------------------------------------------%
	By Lemma \ref{le2},  it is easy to check that for each $G$-free $n$-critical graph $H$, if $F$ is a $G$-free subgraph of $H$, then $\chi_G(H\setminus F)=n-1$.
	%---------------------------------------------------------------------------------------%

	%-----------------------------------------------------------------------------------------
	
	\subsection{Proof of theorem \ref{one-thm1}.}\hfilneg
	
	%------------------------------------------------------------------------------------
	In~\cite{nordhaus1956complementary}, Nordhaus and Gaddum showed that  for each graph $H$ with $n$ vertices, $\chi(H)+\chi(\overline{H})\leq  n+1$.  In the following  results,  we extend  the Nordhaus and Gaddum-problem. To simplify the comprehension, let us split the proof of Theorem \ref{one-thm1} into small parts. We begin with a  very useful general upper bound in the following theorem:
	%---------------------------------------------------------------------------------------%
	%---------------------------------------------------------------------------------------%
	\begin{theorem}\label{vthe1} Let $H$ and $G$ be two graphs, where $\delta(G)=\delta$, $|V(H)|=n$. If $G\cong K_{\delta+1}$, then $n\neq k\delta(k,\delta\geq 3)$, and if $G\ncong K_{\delta+1}$, then $H$ is critical. Hence:
		\begin{itemize}
		
			\item[(I):] $ \chi_G(H)+ \chi_G(\overline{H})\leq \lceil\frac{n}{\delta}\rceil+1 $,
			\item[(II):] This bound is sharp.
		\end{itemize}

	\end{theorem}
	\begin{proof}
		(I).	By induction on $n$, when $n \leq |G|-1$, (I) holds. Suppose
		that $n\geq |V(G)|$ and consider the  following cases:
		
		\bigskip
		{\bf Case1 }: $G\ncong K_{\delta+1}$. Assume that $\chi_G(H)=t_1$ and $\chi_G(\overline{H})=t_2$. We shall show that $t_1+t_2\leq \lceil\frac{n}{\delta}\rceil+1$. As  $H$ is $G$-free $t_1$-critical,  for each $v\in V(H)$, we have $\chi_G(H\setminus\{v\})\leq t_1-1$. So,  Theorem \ref{th1} implies that $\delta(H)\geq (t_1-1)\delta$. Let $v$ be a vertex of $V(H)$. Hence, by induction  assumption we have:   
		\begin{equation}\label{e6}
			\chi_G(H\setminus\{v\})+ \chi_G(\overline{H}\setminus\{v\})\leq \lceil\frac{n-1}{\delta}\rceil+1.
		\end{equation}
		Now we have a  claim as follow:
		
		\bigskip
		{\bf Claim 2}: $\chi_G(\overline{H})\leq\lceil   \frac{n}{\delta}\rceil -t_1+1$.\\
		Since $\chi_G(H\setminus\{v\})\leq t_1-1$ and $\delta(H)\geq (t_1-1)\delta$, it yields that $\Delta(\overline{H})= n-1-\delta(H)\leq n-1-(t_1-1)\delta$. Hence by Corollary \ref{col2}, we have: 
		\[ \chi_G(\overline{H})\leq \lceil\frac{\Delta(\overline{H})+1}{\delta}\rceil\leq \lceil\frac{ n-(t_1-1)\delta}{\delta}\rceil=\lceil\frac{ n}{\delta}\rceil-(t_1-1).\]
		Now, by Claim $2$, $t_1+\chi_G(\overline{H})\leq  \lceil\frac{ n}{\delta}\rceil+1$. Therefore,   $\chi_G(H)+\chi_G(\overline{H})\leq  \lceil\frac{ n}{\delta}\rceil+1$,  and the proof of Case $1$ is complete.
		
		\bigskip
		{\bf Case 2 }: $G\cong K_{\delta+1}$. As $G\cong K_{\delta+1}$ and $n\neq k\delta$, one can check that  $\omega(H)\geq\delta+1$ and  $\omega(\overline{H})\geq \delta+1$, otherwise either $H$ is $G$-free or $\overline{H}$ is $G$-free, then by this fact that $\chi_G(H)\leq \lceil \frac{n(H)}{\delta}\rceil$, the proof is complete. Hence, suppose that $V'=\{v_1,v_2,\ldots,v_{\delta+1}\},V''=\{v'_1,v'_2,\ldots,v'_{\delta+1}\} \subseteq V(H)$, where $H[V']\cong K_{\delta+1}$ and $\overline{H}[V'']\cong K_{\delta+1}$. As, $H[V']\cong K_{\delta+1}$ and $\overline{H}[V'']\cong K_{\delta+1}$, by considering $V'\cap V''$,  one can say that $|V'\cap V''|\leq 1$. Therefore, there exists $W'\subseteq V'$ and $W''\subseteq V''$, so that $|W'|=|W''|=\delta$,  $W'\cap W''=\emptyset$, $H[W']\cong K_{\delta}$ and $\overline{H}[W'']\cong K_{\delta}$. Set $W=W'\cup W''$. Now, we have a claim as follow:
		
		\bigskip
		{\bf Claim 3}:  $|\omega (H[W])|\leq \delta+1$, $|\omega (\overline{H}[W]|\leq \delta+1$ and  $|\omega (H[W])|+|\omega (\overline{H}[W]|\leq  2\delta+1$.\\
		As, $H[W']\cong K_{\delta}$ and $\overline{H}[W'']\cong K_{\delta}$, so $\delta\leq |\omega (H[W])|\leq \delta+1$, $\delta\leq|\omega (\overline{H}[W]|\leq \delta+1$. Now, by contrary, suppose that $|\omega (H[W])|=|\omega (\overline{H}[W]|= \delta+1$. As $|\omega (H[W])|=\delta+1$, there exists a vertex of $W''$ say $v''$, such that $W'\subseteq N_H(v'')$. Since $|\omega  (\overline{H}[W]|=\delta+1$, there exists a vertex of $W'$ say $v'$, such that $W''\subseteq N_{\overline{H}}(v')$, a contradiction to $W'\subseteq N_H(v'')(v'v''\in E(H))$. Hence, $|\omega (H[W])|+|\omega (\overline{H}[W]|\leq  2\delta+1$.\\
		
		Set $H_1=H[V\setminus W]$. Now,  by induction hypothesis:
		\begin{equation}\label{e7}
			\chi_G(H_1)+ \chi_G(\overline{H}_1)\leq \lceil\frac{n(H_1)}{\delta}\rceil+1= \lceil\frac{n-2\delta}{\delta}\rceil+1=\lceil\frac{n}{\delta}\rceil-1.
		\end{equation}
		
		And  by Claim 3:  
		\begin{equation}\label{e8}
			\chi_G(H[W])+\chi_G\overline{H}[W]\leq  3.
		\end{equation} 
		If $|\omega (H[W])|= |\omega (\overline{H})[W]|= \delta$ , then  $\chi_G(H[W])+\chi_G\overline{H}[W]=2$, that is $H[W]$ and $\overline{H}[W]$ are $G$-free. So   \ref{e7} implies that, $\chi_G(H)+\chi_G(\overline{H})\leq \lceil\frac{n}{\delta}\rceil+1$. Hence, we may suppose that $\chi_G(H[W])+\chi_G\overline{H}[W]=3$. 
		If strict inequality holds in  $(\ref{e7})$, then by using $(\ref{e8})$, we have: 
		\begin{equation}\label{e10}
			\chi_G(H)+ \chi_G(\overline{H})\leq\chi_G(H_1)+ \chi_G(\overline{H_1})+3\leq \lceil\frac{n}{\delta}\rceil+1. 
		\end{equation}
		
		Now,  suppose that inequality holds in both $(\ref{e7})$ and $(\ref{e8})$, and by Claim 3, w.l.g suppose that $|\omega (H[W])|=\delta+1$ and $|\omega (\overline{H})[W]|= \delta$. Since, $|\omega (H[W])|=\delta+1$, so there exists at least one vertex of $W''$ say $v''$, such that $W'\subseteq N_H(v'')$. Set $v\in W'$, $H_2=H_1\cup\{v\}$.
		
		Now, since  $|W|=2\delta$ and  $|V(H)|\neq k\delta$, so $|V(H_1)|\neq k'\delta$.	Therefore, as $	\chi_G(H_1)+ \chi_G(\overline{H_1})=\lceil\frac{n}{\delta}\rceil-1$, and $|V(H_1)|\neq k'\delta$, one can check that $\lceil\frac{n}{\delta}\rceil=\lceil\frac{n+1}{\delta}\rceil$. Now,  by the induction hypothesis, $\chi_G(H_2)+\chi_G(\overline{H_2})=\lceil\frac{n}{\delta}\rceil-1$. Since $v\in W'$ and $|\omega (\overline{H}[W]|= \delta$, so   $|\omega (H[W\setminus \{v\}])|=\delta$ and $|\omega (\overline{H}[W\setminus \{v\}]|\leq \delta$. That is:
		\[ 	\chi_G(H[W\setminus \{v\}])+\chi_G\overline{H}[W\setminus \{v\}]=2.\]
		Now, as $\chi_G(H_2)+\chi_G(\overline{H_2})=\lceil\frac{n}{\delta}\rceil-1$, we can check that 	$\chi_G(H)+ \chi_G(\overline{H})\leq\chi_G(H_2)+ \chi_G(\overline{H_2})+2= \lceil\frac{n}{\delta}\rceil+1$, and the proof is complete.
		
		Now by Cases 1, 2 the proof of (I) is complete. To prove (II), consider the  following cases:	
		
		\bigskip
		{\bf Case 1:} $G\ncong K_{\delta+1}$.\\
		In this case, set $H=K_{4,4}$ and $G=H$. As, $G=H$, and $\overline{H}=2K_4$ is $G$-free, thus $\chi_G(H)=2$, $\chi_G(\overline{H})=1$, that is,
		$\chi_G(H)+\chi_G(\overline{H})=2+1=3=\lceil\frac{8}{4}\rceil+1$.
		
		By setting $H=G=C_5$. As, $H=\overline{H}=G=C_5$, so  $\chi_G(H)+\chi_G(\overline{H})=2+2=4=\lceil\frac{5}{2}\rceil+1$.
		
		\bigskip
		{\bf Case 2:}  $G\cong K_{\delta+1}$, and  $|V(H)|\neq k\delta$, $k,\delta\geq 3$.\\
		Consider the graph $H$, where $|V(H)|=kd+1$, $H\cong K_{(k-1)d}+(\delta+1)K_1$ and $\overline{H}\cong K_{\delta+1}$. As $G\cong K_{\delta+1}$, and $\overline{H}\cong K_{\delta+1}$, thus $\chi_G(\overline{H})=2$. Now, we would like to show that $\chi_G(H)=k$. By definition of $H$, one can check that  $|\omega (H[W])|= (k-1)d+1$.  Suppose that, $V=V(H)=\{v_1,v_2,\ldots, v_{k\delta+1}\}$, $V'=\{v_1,\ldots, v_{(k-1)\delta}\}$, where $H[V']\cong K_{(k-1)\delta}$ and $H[V\setminus V']\cong (\delta+1)K_2$. It is easy to say that $\chi_G(H)\leq k$, by considering  $V_1,V_2,\ldots,V_{k}$  as a coloring classes of $V(H)$, where $V_i\subseteq V'$,  $|V_i|=\delta$ for each $1\leq i\leq k-1$ and $V_k= V(H)\setminus V'$. Since $|V_i|=d$, for $1\leq i\leq k-1$, and $\omega (H[V_k])=1$, so $H[V_i]$ is $G$-free, and  $\chi_G(H)\leq k$. Also, since $|\omega (H)|= (k-1)\delta+1$ and $G\cong K_{\delta+1}$, it can be checked that  $\chi_G(H)\geq k$. Therefore, $\chi_G(H)=k$ and $\chi_G(H)+\chi_G(\overline{H})=k+2=	\lceil\frac{n}{\delta}\rceil+1$.  Which implies that  the proof is complete.
	\end{proof}
	With an  argument similar to Claim 3, one can say that  if $H$ and $G$ be two graphs, where $G\cong K_{d+1}$, and  $|V(H)|=2d$, then:
	\[ \chi_G(H)+ \chi_G(\overline{H})\leq 3=\lceil\frac{n}{d}\rceil+1.\]
	Also,  if $H$ and $G$ be two graphs, where $G\ncong K_{d+1}$, and  $|V(G)|= d+2$, then:
	\[ \chi_G(H)+ \chi_G(\overline{H})\leq \lceil\frac{n}{d}\rceil+2.\]
	\begin{lemma}\label{le3}
	 Let $H$ and $G$ be two graphs, where $\delta(G)=\delta$, $|V(H)|=n$, $G\ncong K_{d+1}$, and $H$ is not a critical graph. Hence:
	 \[ \chi_G(H)+ \chi_G(\overline{H})\leq \lceil\frac{n}{\delta}\rceil+2.\]
	\end{lemma}
\begin{proof}
		By induction on $n$, when $n \leq |G|-1$, the lemma holds. Suppose that $n\geq |V(G)|$, $\chi_G(H)=t_1$ and $\chi_G(\overline{H})=t_2$. We shall show that $t_1+t_2\leq \lceil\frac{n}{\delta}\rceil+2$. By Lemma \ref{l1}, $H$ has a $G$-free $t_1$-critical subgraph. Assume that $F$ is a $t_1$-critical subgraph of $H$ with $|V(F)|=n_1\leq n-1$.   Let  $V_1=V(H)\setminus V(F)$ and $|V_1|=n_2$. W.l.g suppose that $V_1=\{v_1,v_2\ldots,v_{n_2}\}$. As $F$ is a $t_1$-critical, so  by Theorem \ref{vthe1},
	
	\begin{equation}\label{e3}
		\chi_G(F)+ \chi_G(\overline{F})\leq \lceil\frac{n_1}{\delta}\rceil+1.
	\end{equation}  
	Since $F$ is $G$-free $t_1$-critical subgraph of $H$, then $\chi_G(H)=\chi_G(F)=t_1$. Let $\chi_G(\overline{H}[V(F)])=t_3$. Hence: 
	\begin{equation}\label{e4}
		\chi_G(F)+ \chi_G(\overline{F})= \chi_G(H)+ \chi_G(\overline{H}[V(F)])=t_1+t_3\leq \lceil\frac{n_1}{\delta}\rceil+1. 
	\end{equation}

	We note that $ \chi_G(\overline{H})\leq \chi_G(\overline{H}[V(F)])+\chi_G(\overline{H}[V_1])$. Now  we can partition the vertices of $V_1$ into $m$ sets with size at most $\delta+1$, say $W_1,W_2,\ldots,W_m$. Since $G\ncong K_{\delta+1}$, thus $|V(G)|\geq \delta+2$,  that is $\overline{H}[W_i]$ is $G$-free. So, $\chi_G(\overline{H}[V_1])\leq \lceil  \frac{n_2}{\delta+1}\rceil \leq \lceil   \frac{n_2}{\delta}\rceil$. Hence $ \chi_G(\overline{H})\leq t_3+\chi_G(\overline{H}[V_1])$. That is:
	
	\[ \chi_G(H)+ \chi_G(\overline{H})\leq \lceil\frac{n}{\delta}\rceil+2 .\]

	Which means that the proof is complete.
\end{proof}
	In the next theorem, we determine an upper bound for  $\chi_G(H)+ \chi_G(\overline{H})$, where $G\cong K_{d+1}$, and  $|V(H)|=kd (k, d\geq 3)$. Also, we characterize certain classes of graphs with better upper bound.
	\begin{theorem}\label{vthe2} Suppose that $H$ and $G$ be two graphs, where $G\cong K_{d+1}$, and  $|V(H)|=kd(k,d\geq 3)$, then:
		\[ \chi_G(H)+ \chi_G(\overline{H})\leq k+2.\]
		And if:
		\begin{itemize}
			\item[(I):]  Either, $H$ or $\overline{H}$ is $G$-free critical. 
			\item[(II):] There exists a subset of $V(H)$ say $S$, with  size $2d$, such that $H[S]$ and $\overline{H}[S]$ is $G$-free.
			\item[(III):] There exists a subset of $V(H)$ say $S$, with  size $2d$, such that either $g(H[S])=2d$ or $g(\overline{H}[S])=2d$.
		\end{itemize}
		Where, $g(H)$ is the girth of $H$ and is defined as the size of the smallest cycle in  $H$. Then:
		\[ \chi_G(H)+ \chi_G(\overline{H})\leq k+1.\]
	\end{theorem}
	\begin{proof}
		With an  argument similar to Theorem \ref{vthe1}, one can say that:
		\[ \chi_G(H)+ \chi_G(\overline{H})\leq k+2.\]
		Now, to prove (I), w.l.g we may suppose that $H$ is $G$-free critical, that is for each $v\in V(H)$, $\chi_G(H\setminus\{v\})\leq \chi_G(H) -1$. Let $v$ be a vertex of $V(H)$. Hence, by Theorem \ref{vthe1},  
		\begin{equation}
			\chi_G(H\setminus\{v\})+ \chi_G(\overline{H}\setminus\{v\})\leq \lceil\frac{n-1}{\delta}\rceil+1.
		\end{equation}
		As, $H$ is  $G$-free critical, $\delta(H)\geq (\chi_G(H) -1)\delta$, it yields $\Delta(\overline{H})= n-1-\delta(H)\leq n-1-(\chi_G(H) -1)\delta$. Hence, by Corollary \ref{col2}, 
		\[ \chi_G(\overline{H})\leq \lceil\frac{\Delta(\overline{H})+1}{\delta}\rceil\leq \lceil\frac{ n-(\chi_G(H) -1)\delta}{\delta}\rceil=\lceil\frac{ n}{\delta}\rceil-(\chi_G(H) -1).\]
		Therefore, $\chi_G(H)+\chi_G(\overline{H})\leq  \lceil\frac{ n}{\delta}\rceil+1$  and the proof of (I) is complete.
		
		We prove (II) by induction on $k$. For $k=3$, it is clear that $\chi_G(H)+\chi_G(\overline{H})\leq  k+1$. For $k=4$, as $H[S]$, $\overline{H}[S]$ are $G$-free, and $|S|=2k$, so $\chi_G(H[S])+\chi_G(\overline{H})[S]=2$. Hence, with an  argument similar to Claim 3, either $H[\overline{S}]$ is $G$-free or $\overline{H}[\overline{S}]$ is $G$-free. So, $\chi_G(H[\overline{S}])+\chi_G(\overline{H})[\overline{S}]\leq 3$ and 
		$\chi_G(H)+\chi_G(\overline{H})\leq 5=k+1$. Now, let for each $k'\leq k-1$ the statement holds, and suppose that $H$ be a graph with $kd$ vertices. W.l.g assume that $S$ with  size $2d$  be a subset of $V(H)$, where $H[S]$ and $\overline{H}[S]$ are $G$-free. Set $H'=H\setminus S$, Hence, $|H'|=(k-2)d$, and by induction,   $\chi_G(H')+\chi_G(\overline{H'})\leq k-1$. As,  $H[S]$ and $\overline{H}[S]$ are $G$-free, thus $\chi_G(H)+\chi_G(\overline{H})\leq k+1$.
		
		To prove (III), w.l.g we may suppose that  $S$ with  size $2d$  is a subset of $V(H)$, such that  $g(H[S])=2d$. As, $g(H[S])=2d$, $d\geq 3$, and $G=K_{d+1}$, so $H[S]$ is $G$-free. Since  $g(H[S])=2d$, so $\overline{H}[S]\cong K_{2d}\setminus C_{2d}$, and  $\omega(\overline{H}[S])=d$, that is $\overline{H}[S]$ is $G$-free. Therefore, by proof of (II), the proof of (III) is complete.
		
		Which implies that  the proof is complete.
	\end{proof}
	%%%%%%%%%%%%%%%%%%%%%%%%%%%%%%%%%%%%
Assume that   $\C$  is a family of all $2$-regular graphs.	In the next theorem, we determine an upper bound for  $\chi_G(H)+ \chi_G(\overline{H})$, for each $G\in \C$.
	
	\begin{theorem}\label{vthe3} Let $H$ and $G$ be two graphs, where $G\in  \C$, hence:
		\[ \chi_{G}(H)+ \chi_{G}(\overline{H})\leq \lceil\frac{n}{2}\rceil+1.\]
		And this bound is sharp. 
	\end{theorem}
	\begin{proof}
	{\bf Case 1}: $G=K_3$. If $|V(H)\neq 2k$ for each $k\geq 3$, then the proof is complete by Theorem \ref{vthe1}. So, suppose that $|V(H)= 2k$ for some $k\geq 1$. 	We prove by induction. For $k=1,2$ it is clear that $\chi_{K_3}(H)+ \chi_{K_3}(\overline{H})\leq k+1$. Suppose that $\chi_{K_3}(H)+ \chi_{K_3}(\overline{H})\leq k'+1$, for each $k'\leq k-1$, and assume that $H$ be a graph with $2k$ vertices, where $k\geq 3$. Assume that $v,v'$ be two vertices of $H$, where $\deg_H(v)=\deg_H(v')$. Set $H'=H\setminus\{v,v'\}$, as $|V(H')|=2(k-1)$, by induction  $\chi_{K_3}(H')+ \chi_{K_3}(\overline{H'})\leq k$. If $\chi_{K_3}(H')+ \chi_{K_3}(\overline{H'})\leq k-1$, then the proof is complete. Hence, assume that $\chi_{K_3}(H')+ \chi_{K_3}(\overline{H'})= k$. If either $\chi_{K_3}(H')= \chi_{K_3}(H)$ or  $\chi_{K_3}(\overline{H'})=\chi_{K_3}(\overline{H})$, then the proof is same. Now, suppose that  $\chi_{K_3}(H)= \chi_{K_3}(H')+1$ and  $\chi_{K_3}(\overline{H})=\chi_{K_3}(\overline{H'})+1$. In other word, suppose that adding $x$ to $H'$ and $x'$ to $\overline{H'}$  increases the $\chi_{K_3}(H')$ and  $\chi_{K_3}(\overline{H'})$ by one, where $x,x'\in \{v,v'\}$. Now we have a claim as follow:

		\bigskip
		{\bf Claim 4}:  $x\neq x'$. By contrary, suppose that $x=x'$, $\chi_{K_3}(H')=t_1$ and  $\chi_{K_3}(\overline{H'})=t_2$, where $t_1+t_2=k$. As $\chi_{K_3}(H'\cup \{x\})=t_1+1$ and  $\chi_{K_3}(\overline{H'}\cup \{x\})=t_2+1$, we can say that $|N_H(x)|\geq 2t_1$ and $|N_{\overline{H}}(x)|\geq 2t_2$. Now,  $2k-1=|N_H(x)|+|N_{\overline{H}}(x)|\geq 2(t_1+t_2)$, that is $2k\leq 2k-1$, a contradiction.\\
		
		So, suppose that $x=v$ and $x=v'$.  Since $\chi_{K_3}(H'\cup \{v\})=t_1+1$ and  $\chi_{K_3}(\overline{H'}\cup \{v'\})=t_2+1$, so $|N_H(v)|\geq 2t_1$ and $|N_{\overline{H}}(v')|\geq 2t_2$. As  $\deg_H(v)=\deg_H(v')$, thus  $n-1=2k-1=|N_H(v)|+|N_{\overline{H}}(v')|\geq 2(t_1+t_2)$, that is  $2k\leq 2k-1$, a contradiction again. Hence, we can say that $\chi_{K_3}(H')+ \chi_{K_3}(\overline{H'})\leq k-1$, or $\chi_{K_3}(H')= \chi_{K_3}(H)$, or  $\chi_{K_3}(\overline{H'})=\chi_{K_3}(\overline{H})$, in any case, $\chi_{K_3}(H)+ \chi_{K_3}(\overline{H})\leq k+1$.
		
		{\bf Case 2}: $G=C_m$, $m\geq 4$. For the case that $|V(H)| =2k$, the proof is same as the Case 1. Hence assume that $|V(H)|=2k+1$ for some $k\geq 2$. We prove by induction. For $k=1,2$ it is clear that $\chi_{C_m}(H)+ \chi_{C_m}(\overline{H})\leq k+1$. Suppose that $\chi_{C_m}(H)+ \chi_{C_m}(\overline{H})\leq k'+1$ for each $k'\leq k-1$, and assume that $H$ be a graph with $2k+1$ vertices, where $k\geq 3$. Assume that $V'=\{v_1,v_2,v_3\}$ be a subsets of $V(H)$. Set $H'=H\setminus V'$, as $|V(H')|=2(k-1)$, by induction,  $\chi_{C_m}(H')+ \chi_{C_m}(\overline{H'})\leq k$. Also as $|V(G)|=m\geq $, so $H[V']$ and $\overline{H}[V']$ are $G$-free. Therefore, $\chi_{C_m}(H)+ \chi_{C_m}(\overline{H})\leq \chi_{C_m}(H')+1+ \chi_{C_m}(\overline{H'})+1\leq k+2= \lceil\frac{n}{2}\rceil+1$,  and the proof of Case 2 is complete.

		To prove the sharpness of the bound, set $H=K_3+3K_1$, that is $\overline{H}=K_3$. Hence,  $\chi_{K_3}(H)=\chi_{K_3}(\overline{H})=2$. Therefore, as $|V(H)|=6$, so $\chi_{K_3}(H)+\chi_{K_3}(\overline{H})=4=\lceil\frac{6}{2}\rceil+1$. Which implies that  the proof is complete.
		
	\end{proof}
	
	%%%%%%%%%%%%%%%%%%%%%%%%%%%%%%%%%%%%%%%%%%%%%%%%%%%%%%%%%%%%%%%%%%%%%%%%%%%%%%%%%%%%%%%%%%%%%%%%%%%%%%%%%%%%%%%%%%%%%%%%%%

	Therefore by Theorems \ref{vthe1}, \ref{vthe2}, and \ref{vthe3}, we have a result as follow:
	%%%%%%%%%%%%%%%%%%%%%%%%%%%%%%%%%%%%%% in dorste
	\begin{theorem} 
		Suppose that $\G$
		is a family of  graphs with  minimum degrees  $\delta(\G)$, where $\delta(\G)=\min\{\delta(G):~G\in\G \}$. Also, let $H$ be a  graph with $n(H)$ vertices.  Then:
		
		\[
		\chi_{\G}(H)+ \chi_{\G}(\overline{H})\leq\left\lbrace
		\begin{array}{ll}
			
			\lceil\frac{n(H)}{\delta(\G)}\rceil+1 & ~~~~
		when ~\G=\C ~or	~either,~~ H~ or~\overline{H}~is ~\G-free~ critical, ~\vspace{.2 cm}\\
			\lceil\frac{n(H)}{\delta(G)}\rceil+2 & ~~~~ ~~~~~~ otherwise.
		\end{array}
		\right.
		\]	
	\end{theorem}
	\begin{proof}
		Suppose that $G\in \G$, clearly $\chi_{\G}(H) \leq	\max_{G\in\G}\chi_{G}(H)$. Assume that $\max_{G\in\G}\chi_{G}(H)=	\chi_{G'}(H)$, and $\delta'=\delta(G')$. Hence, $\chi_{\G}(H) \leq\chi_{G'}(H)$. Since $\delta(\G)\leq \delta'$ and by Theorems \ref{vthe1}, \ref{vthe2} and \ref{vthe3}, one can check that $\chi_{\G}(H) \leq\chi_{G'}(H)\leq 	\lceil\frac{n(H)}{\delta'}\rceil+i\leq	\lceil\frac{n(H)}{\delta(\G)}\rceil+i$, where $i\in \{1,2\}$. Which implies that  the proof is complete.

	\end{proof}
	In Theorem~\ref{one-thm1}, if we take 
	$\G=\{K_2\}$, then we get  Nordhaus-Gaddum's result in terms of chromatic number. Also, assume that   $\C$  is a family of all $2$-regular graphs. One can see for a given graph $H$,  $a(H)$ and $\chi_{\C}(H)$ are  identical. Therefore, if we take $\G=\C$,  then by Theorem \ref{one-thm1}, we get  Theorem~\ref{tho1}.
	%%%%%%%%%%%%%%%%%%%%%%%

	%========================================================================================
	%%%%%%%%%%%%%%%%%%%%%%%%%%%%%%%%%%%%%%%%%
	\bibliographystyle{plain}
	\bibliography{G-free}
\end{document}